\documentclass{amsart}
\usepackage{xcolor,enumerate}       
 \definecolor{darkgreen}{rgb}{0,0.5,0}
\usepackage{blindtext}
\usepackage{mathtools}
\usepackage{tikz}
\usetikzlibrary{arrows.meta,calc,positioning,decorations.pathmorphing}

\usepackage[T1]{fontenc}    
\usepackage{hyperref}       
\usepackage{url}          

\usepackage[margin=1in]{geometry}
\usepackage[T1]{fontenc}
\usepackage[utf8]{inputenc}
\usepackage{lmodern}
\usepackage{microtype}
\usepackage{comment}
\usepackage[ruled,vlined]{algorithm2e}

\usepackage{amsmath,amssymb,amsthm,amsfonts,mathtools,bm}
\usepackage{mathrsfs}
\usepackage{enumitem}
\usepackage{tikz-cd}
\usepackage{bbm}
\usepackage{dsfont}
\usepackage{hyperref}
\usepackage[nameinlink,capitalize]{cleveref}
\usepackage{xcolor}
\usepackage{graphicx}     
\usepackage{subcaption}   
\usepackage{caption}
\usepackage{graphicx} 
\newtheorem{theorem}{Theorem}[section]
\newtheorem{proposition}[theorem]{Proposition}
\newtheorem{lemma}[theorem]{Lemma}

\theoremstyle{definition}
\newtheorem{definition}[theorem]{Definition}
\newtheorem{assumption}[theorem]{Assumption}

\newtheorem{example}[theorem]{Example}

\newcommand{\R}{\mathbb{R}}

\newcommand{\KL}{\mathrm{KL}}

\newcommand{\tr}{\mathrm{tr}}

\newcommand{\diver}{\mathrm{div}}
\newcommand{\ol}{\overline}
\newcommand{\eps}{\varepsilon}
\newcommand{\Om}{\Omega}
\newcommand{\p}{\partial}

\newcommand{\norm}[1]{\left\lVert #1\right\rVert}

\newcommand{\dd}{\,\mathrm{d}}
\newcommand{\map}[3]{#1:#2 \rightarrow #3}

\title{ Reflected Schr\"odinger Bridge Problem
over Noncharacteristic Sub-Riemannian Manifolds}
\author{Daniel Owusu Adu}
\date{}

\begin{document}

\begin{abstract}
We formulate and solve the reflected sub-Riemannian Schr\"odinger problem: minimum-energy transport of probability distributions for underactuated dynamics with degenerate diffusion on a bounded domain. Since Euclidean normal reflection is generally incompatible with the horizontal subbundle, we construct an intrinsic reflection field from the geometry of the diffusion; it is well-defined when the boundary is non-characteristic, structurally delimiting the theory. Under appropriate assumptions, we establish the existence of a smooth, positive transition density for the reference law, and the optimal control is characterized by a mild Schr\"odinger system with asymmetric boundary conditions, determined by the reflection field, Popp measure, sub-Riemannian fields and It\^o drift. This characterization yields a PDE-based Sinkhorn iteration that enforces these conditions directly, without evaluating the unavailable transition kernel.
\end{abstract}
\maketitle

\section{Introduction} 
Motivated by the task of redistributing an ensemble of underactuated agents within a confined workspace while minimizing energy~\cite{caluya2021reflected}, and by generative modeling
problems where transport must respect data domains with boundaries~\cite{deng2024reflected}, we study a Schr\"odinger bridge problem (SBP)~\cite{schrodinger1931umkehrung}, where the diffusion is constrained on a sub-Riemannian manifold $(\overline{\mathcal{X}},\mathrm D,\mathrm d)$. Here $\mathcal{X} := \{x \in \mathbb{R}^d : \xi(x) > 0\}$ is an open bounded domain with boundary $\partial\mathcal{X} := \{x \in \mathbb{R}^d : \xi(x) = 0\}$, where $\xi \in C^{\infty}(\mathbb{R}^d)$.  Throughout, we assume that $\nabla\xi \neq 0$ on 
$\partial\mathcal{X}$ so that the inward normal is 
\begin{equation}\label{eq: normal unit vector}
n(x) := \nabla\xi(x)/\|\nabla\xi(x)\|.  
\end{equation}
We endow $\overline{\mathcal X}$ with an equiregular (see~\cite[p.~559]{agrachev2019comprehensive} for the definition) horizontal subbundle
\begin{equation}\label{eq: horizontal subbundle}
\mathrm{D}(x) := \operatorname{span}\{g_1(x), \dots, g_m(x)\} \subset \mathbb{R}^d
\end{equation}
together with the sub-Riemannian metric $\mathrm d_x(v,v) = v^\top G(x)^\dagger v$, for \(v\in \mathrm D(x)\) and $x\in\overline{\mathcal{X}}$. Here $G^\dagger(x)=(g(x) g(x)^\top)^\dagger$ denotes the Moore-Penrose pseudoinverse of $G(x)$ at $x\in\overline{\mathcal{X}}$, where  $\map{g}{\overline{\mathcal{X}}}{\mathbb{R}^{d \times m}}$ with $m < d$ has column $g_i$ in~\eqref{eq: horizontal subbundle}. Let  $\Omega := C([0,1]; \overline{\mathcal{X}})$ denote the path space and  $\mu_0,\mu_f\in\mathcal{P}(\overline{\mathcal{X}})$ admit densities $\rho_0,\rho_f$ with respect to Popp measure $\mathrm{m}_{\mathrm P}$. For $\epsilon>0$, let $\mathrm{R}^{\epsilon,\mu_0}\in \mathcal{P}(\Omega)$  be a reference measure supported on 
\begin{equation}\label{eq:reference_SDE}
dX_t = b_\epsilon(X_t)dt + \sqrt{\epsilon}\,g(X_t)dW_t + r(X_t)d\eta_t,~~ X_0 \sim \mu_0,
\end{equation}
whenever the solution exists. Here: $X_t \in \mathbb{R}^d$ is the state of the process at time $t$; $b_\epsilon(x) := \frac{\epsilon}{2}\sum_{i=1}^m \nabla_{g_i}g_i(x)$ is the Stratonovich-to-It\^o correction; $\{W_t\}_{t \ge 0} \subset \mathbb{R}^m$ is the standard Brownian motion;  $\eta_t$ is the boundary regulating process~\cite{Pilipenko2014}. It is a continuous, nondecreasing process satisfying $\eta_0=0$ and $\eta_t=\int_0^t \mathbf 1_{\partial\mathcal X}(X_s)\,d\eta_s$. The process
$\eta_t$ increases only when $X_t\in\partial\mathcal X$; and $\map{r}{\partial\mathcal{X}}{\R^d}$ is the reflection term. Thus, for any Borel set $A\subset \overline{\mathcal{X}}$, we have that $((\mathrm{ev}_0)_\# \mathrm{R}^{\epsilon,\mu_0})(A):=\mathrm{R}^{\epsilon,\mu_0}(\mathrm{ev}_0^{-1}(A))=\mu_0(A)$,
where $\map{\mathrm{ev}_t}{\Omega}{\overline{\mathcal{X}}}$ is the evaluation map at $t \in [0,1]$. The SBP for our setup is
\begin{equation*}
\noindent\textbf{P1}:\inf_{\mathrm{P} \in \mathcal{C}_{\mathrm{SB}}(\mathrm{R}^{\epsilon,\mu_0},\mu_f)} \, \KL(\mathrm{P} \| \mathrm{R}^{\epsilon,\mu_0})
 =\mathbb{E}_{\mathrm{P}}\!\left[ \log \frac{d\mathrm{P}}{d\mathrm{R}^{\epsilon,\mu_0}} \right],
\end{equation*}
where $\mathcal{C}_{\mathrm{SB}}(\mathrm{R}^{\epsilon,\mu_0},\mu_f)
:= \big\{\, \mathrm{P} \in \mathcal{P}(\Omega) : 
\mathrm{P} \ll \mathrm{R}^{\epsilon,\mu_0}, \;\;
\\(\mathrm{ev}_0)_\# \mathrm{P} = \mu_0, \;\;
(\mathrm{ev}_{1})_\# \mathrm{P} = \mu_f \,\big\}$.
Here, 
the absolute-continuity constraint $\mathrm{P} \ll \mathrm{R}^{\epsilon,\mu_0}$ ensures that the Radon--Nikodym derivative $d\mathrm{P}/d\mathrm{R}^{\epsilon,\mu_0}$ exists, and the pushforward conditions enforce the prescribed end-time marginals. From~\cite{chen2016relation,dai1991stochastic,Leonard2014Survey}, by Girsanov transformation~\cite[Chapter~8.6]{oksendal2003stochastic}, Problem~$1$ is equivalent to solving
\begin{equation}\label{eq: into_minimum_energy}
\noindent\textbf{P2}:~~~\inf_{u \in \mathcal{U}}\mathbb{E}_{\mathrm{P}}\left[\int_0^{1} \frac{1}{2\epsilon}\|u(t,X_t)\|^2 dt\right],
\end{equation}
where $\mathrm{P}$ in Problem~$1$ is supported on
\begin{equation}\label{eq:controlled_reflected_SDE}
dX_t = \bigl(b_\epsilon(X_t) + g(X_t)u(t,X_t)\bigr)dt 
+ \sqrt{\epsilon}\,g(X_t)dW_t + r(X_t)d\eta_t,
\quad X_0 \sim \mu_0\quad\text{satisfies}\quad X_{1} \sim \mu_f,
\end{equation}
where $\mathcal{U} := \{\map{u}{[0,1]\times\overline{\mathcal{X}}}{\R^m} : \mathbb{E}[\int_0^{1}\|u(t,X_t)\|^2 dt] < \infty\}$.

Since removing a constant factor preserves minimizers, we will drop $\frac{1}{\epsilon}$ in Problem~$2$. Our setup~\eqref{eq: horizontal subbundle}-\eqref{eq: into_minimum_energy} unifies two independent frameworks: reflected Schr\"odinger bridges, where hard state constraints are
enforced on $\overline{\mathcal{X}}$ by a reflection mechanism~\cite{caluya2021reflected,deng2024reflected,kalise2025pde} and sub-Riemannian Schr\"odinger bridges, where the
diffusion is anisotropic, modeling underactuated systems \(m<d\)
in~\cite{adu2026schrodinger_manifold}.  Beyond this unification, leading to two fundamental constraints in underactuated systems: hard state confinement and non-holonomic kinematics, complications arise. Previous reflected Schr\"odinger bridge
formulations in~\cite{caluya2021reflected,deng2024reflected,kalise2025pde} typically assume isotropic diffusion and a normal reflection mechanism $r=n$ in~\eqref{eq: normal unit vector}. These
assumptions are natural in elliptic settings, but
in general, $n\notin\mathrm{D}$ in~\eqref{eq: horizontal subbundle}. 
Second, even if the reflection mechanism $r$ in~\eqref{eq:reference_SDE} satisfies $r\in\mathrm{D}$, without any extra assumptions on $\mu_0$ and $\mu_f$, Problem~$1$ can be ill-posed (see Figure~\ref{fig:uncontrolled_comparison}).

To overcome these difficulties, we construct a reflection mechanism $r$ that is intrinsic to $(\overline{\mathcal{X}},\mathrm D,\mathrm d)$ through an optimization criterion. This is the unique maximal inward oblique reflection  $r\in\mathrm D$ in~\eqref{eq: horizontal subbundle}. 
This construction is well-defined when the boundary is non-characteristic. In addition to
H\"ormander's bracket-generating condition~\cite{Hormander1967}, following from~\cite{cattiaux1988time}, we show that the reflected anisotropic
diffusion~\eqref{eq:reference_SDE} admits a smooth, strictly positive transition density on $\overline{\mathcal{X}}\times \overline{\mathcal{X}}$. Consequently, following from~\cite{Fortet1940}, Problem~$1$ admits a unique solution. Although this transition density is central in the construction of the solution for Problem~$1$ and~$2$, explicit closed-form formulas are
generally unavailable outside special low-dimensional or highly
symmetric cases~\cite{hsu1986excursions,lou2023reflected}. Therefore, a kernel-based Sinkhorn algorithm, central to the numerical scheme, is
not practical in the present setting. We instead formulate a PDE-based
Sinkhorn scheme. The novel structural feature of this scheme is that
the two Schr\"odinger factors satisfy different boundary conditions:
the backward factor satisfies the oblique Neumann condition associated
with the intrinsic reflection mechanism, while the forward factor satisfies
the adjoint no-flux condition enforcing conservation of probability
mass.

Finally, while obliquely reflected degenerate diffusions have been
studied in the stochastic analysis literature
~\cite{cattiaux1988time,chassagneux2020obliquely,ElKaroui1997}, the reflection fields in those works are typically
prescribed a priori. In contrast, the reflection field introduced here
is geometrically determined by the sub-Riemannian diffusion itself. It is the unique horizontal direction that is maximally inward pointing. SB have been extended in other directions
~\cite{adu2026schrodinger,bondar2026nonlinear,caluya2021wasserstein}. Besides~\cite{caluya2021reflected,deng2024reflected} the work in~\cite{bondar2026nonlinear} again treats
non-degenerate noise without boundary, though with a mismatch between the noise and control
channels handled algorithmically.  To the best of our knowledge, this work is the first to consider SB of sub-Riemannian diffusions with boundary. It is treated geometrically via the horizontal subbundle and reflection field.
\section{well-posedness and existence of solution}\label{sec:problem}
To ensure that the Problems~$1$ and~$2$ are well-posed, throughout we impose the following assumptions:
\begin{assumption}\label{eq: regularity_on_density}
The densities $\rho_0,\rho_f \in L^1(\mathrm{m}_{\mathrm P})$ are continuous and strictly positive on
$\overline{\mathcal X}$.
\end{assumption}

\begin{assumption}\label{ass:regularity}
For each $i=1,\ldots,m$, the vector field $g_i \in C^\infty(\overline{\mathcal{X}};\mathbb{R}^d)$.
\end{assumption}

\begin{assumption}\label{ass:hormander} 
For every $x \in \overline{\mathcal{X}}$, the Lie algebra generated by the vector fields spans the tangent space $\operatorname{Lie}(g_1,\ldots,g_m)(x) = \mathbb{R}^d$.
\end{assumption}

\begin{assumption}\label{ass:obliqueness}
For every $x \in \partial\mathcal{X}$, let $\mathrm T_x\partial\mathcal{X}$ be the tangent space at $x\in \partial\mathcal{X}$. We assume that, for all $x \in \partial\mathcal{X}$, $\mathrm D$ in~\eqref{eq: horizontal subbundle} satisfies $\mathrm D(x)\not\subset \mathrm T_x\partial\mathcal X$.
\end{assumption}

\begin{figure}[!t]
\centering
\includegraphics[width=0.25\columnwidth]
{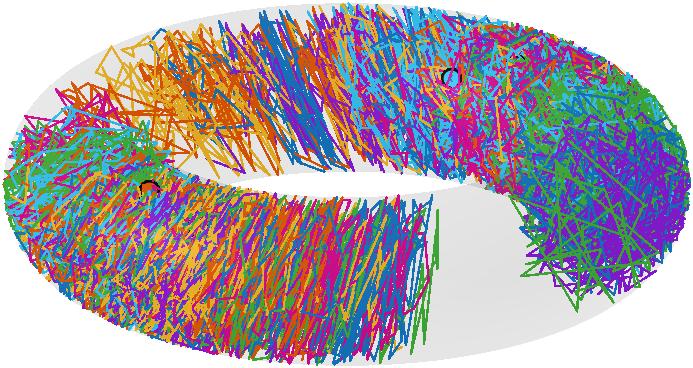}%
\hspace{0.2\columnwidth}%
\includegraphics[width=0.25\columnwidth]
{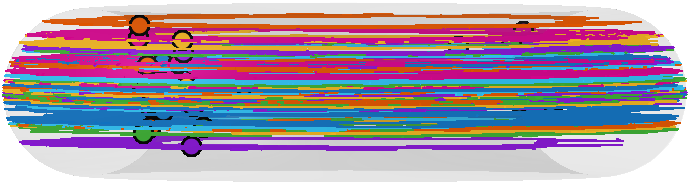}
\caption{Uncontrolled obliquely reflected diffusions
in~\eqref{eq:reference_SDE}: bracket-generating (left)
and confined to horizontal slices (right).
For the latter, any $\mathrm P\ll\mathrm R^{\epsilon,\mu_0}$
with initial law $\mu_0$ preserves its slice marginal,
excluding targets with a different slice marginal.}
\label{fig:uncontrolled_comparison}
\end{figure}
Geometrically, Assumption~\ref{ass:obliqueness} is equivalent to $g(x)^\top n(x)\neq 0$ for all $x\in \partial\mathcal{X}$. Under Assumption~\ref{ass:obliqueness} the optimization problem  $\sup_{\|a\| = 1} \langle g(x)a, n(x) \rangle$
admits the unique solution $a_r(x)=\frac{g(x)^\top n(x)}{\|g(x)^\top n(x)\|}$. Thus,
\begin{equation}\label{eq: reflection field}
r(x) :=  \frac{G(x) n(x)}{\|g(x)^\top n(x)\|}
\end{equation}
is well-defined. The above  reflection is intrinsic to the 
sub-Riemannian structure. Indeed, among all horizontal fields of the form 
$g(x)a$ with $\|a\|=1$,  $r$ in~\eqref{eq: reflection field} is horizontal and maximally inward-pointing. For the special case where the diffusion matrix is $g(x) = I_d$, we get that~\eqref{eq: reflection field} reduces to~\eqref{eq: normal unit vector}.  Thus, the reflection field $r$ in~\eqref{eq: reflection field} is oblique~\cite{cattiaux1988time,chassagneux2020obliquely,ElKaroui1997}.
The existence of solution for Problem~$1$ and~$2$ relies on first establishing the non-emptiness of the support of the reference law $\mathrm{R}^{\epsilon,\mu_0}$ for~\eqref{eq:reference_SDE}. 
\begin{proposition}
\label{pro:wellposedness}
Under Assumptions~\ref{ass:regularity} and~\ref{ass:obliqueness}, the SDE~\eqref{eq:reference_SDE} with~\eqref{eq: reflection field} admits a weak solution $\{X_t\}_{t\geq 0}$ and the law $\mathrm{R}^{\epsilon,\mu_0}$  is unique.
\end{proposition}
\begin{proof}
Since \(\partial\mathcal X\) is \(C^\infty\), the boundary can be locally
flattened and hence we can identify
\(\overline{\mathcal X}\) locally with the upper half-space. Hence we analyze the reflected
SDE~\eqref{eq:reference_SDE} locally in half-space coordinates.
For this, we verify the hypotheses~\cite[Chapter IV, Section 7, Theorem~7.2]{ikeda2014stochastic} after flattening
the boundary by local charts. To this end, we set $\sigma(x):=\sqrt{\epsilon}\,g(x)$ and $a(x):=\epsilon G(x)$. Then, by Assumption~\ref{ass:regularity} 
we have that \(\sigma\) and $b_\epsilon
=
\frac{\epsilon}{2}\sum_{i=1}^m \nabla_{g_i}g_i$ are
 Lipschitz on \(\overline{\mathcal X}\).
By Assumption~\ref{ass:obliqueness}, since $g^\top n\neq 0$ and smooth, we have that $\inf_{x\in\partial\mathcal X}
\|g(x)^\top n(x)\|
\ge \kappa>0$.
Therefore, $r\cdot n
=
\frac{n^\top Gn}
{\|g^\top n\|}
=
\|g^\top n\|
\ge \kappa>0$.
Moreover, for all $x\in\partial\mathcal{X}$,  we have that 
\begin{equation}\label{eq: non-characteristic with boundary}
n(x)^\top a(x)n(x)
=
\epsilon \|g(x)^\top n(x)\|^2
\ge
\epsilon\kappa^2>0.    
\end{equation}
This completes the proof.
\end{proof}
Next we establish that $\mathrm{R}^{\epsilon,\mu_0}$ for~\eqref{eq:reference_SDE} admits a positive transition density $p^r_{t,\epsilon}\in C^\infty(\overline{\mathcal X}\times \overline{\mathcal X})$ for $t>0$. This is established through the infinitesimal generator of~\eqref{eq:reference_SDE}.

\begin{lemma}
The generator for~\eqref{eq:reference_SDE} with~\eqref{eq: reflection field} is
\begin{equation}\label{eq: infinitesimal_generator}
\mathcal L_\epsilon
=
b_\epsilon\cdot\nabla
+
\frac{\epsilon}{2}\operatorname{tr}(G\nabla^2)    
\end{equation}
with domain
\begin{equation}\label{eq: domain_infinitesimal_generator}
\mathrm{Dom_r}
=
\left\{
\varphi\in C^2(\overline{\mathcal X})
:
r\cdot\nabla\varphi=0
\text{ on }\partial\mathcal X
\right\}.
\end{equation}
Moreover, on $\mathcal X$, it can be expressed as
\begin{equation}\label{eq:reflected_sum_of_squares}
\mathcal L_\epsilon
=
\frac12\sum_{i=1}^m V_i^2,\quad\text{where}\quad V_i:=\sqrt{\epsilon}\,g_i\cdot\nabla
\end{equation}
and the square here means composition.
\end{lemma}
\begin{proof}
Applying It\^o's formula to $\varphi(X_t)$, where $\varphi\in \mathrm{Dom_r}$ and $X_t$ solves~\eqref{eq:reference_SDE}, gives
$d\varphi(X_s)
=
\mathcal L_\epsilon\varphi(X_s)\,ds +
\sqrt{\epsilon}\,\nabla\varphi(X_s)^\top g(X_s)\,dW_s+
\nabla\varphi(X_s)\cdot r(X_s)\,d\eta_s$,
where $\mathcal L_\epsilon$ is in~\eqref{eq: infinitesimal_generator}. Since  $d\eta_t$ is supported on 
$\partial\mathcal X$ and  $\varphi\in \mathrm{Dom_r}$, integrating both sides over $[0,t]$, yields
$\varphi(X_t)-\varphi(X_0)
-
\int_0^t
\mathcal L_\epsilon\varphi(X_s)\,ds=\sqrt{\epsilon}\int_0^t\nabla\varphi(X_s)^\top g(X_s)\,dW_s$.
Hence, under Assumption~\ref{ass:regularity} we get that $\varphi(X_t)-\varphi(X_0)
-
\int_0^t
\mathcal L_\epsilon\varphi(X_s)\,ds$ is a martingale. Thus, the reflected generator acts as~\eqref{eq: infinitesimal_generator}
on $\mathrm{Dom_r}$ in~\eqref{eq: domain_infinitesimal_generator}. Furthermore, direct computation yields $\frac12\sum_{i=1}^m V_i^2=b_\epsilon\cdot\nabla
+
\frac{\epsilon}{2}\operatorname{tr}(G\nabla^2)$.
\end{proof}
We proceed to prove that $\mathrm{R}^{\epsilon,\mu_0}$ for~\eqref{eq:reference_SDE} admits a smooth positive transition density on $\overline{\mathcal{X}}\times \overline{\mathcal{X}}$.

\begin{proposition}\label{lem:kernel}
The unique reference measure $\mathrm{R}^{\epsilon,\mu_0}\in\mathcal P(\Omega)$  associated to~\eqref{eq:reference_SDE} with~\eqref{eq: reflection field} 
admits a smooth positive transition density function $p^r$ on $\overline{\mathcal{X}}\times \overline{\mathcal{X}}$. 
\end{proposition}
\begin{proof}
We verify the hypotheses of
\cite[Theorem~2.9]{cattiaux1992stochastic}.
Define $\psi=\xi/\|\nabla\xi\|$ in a boundary neighborhood
where $\nabla\xi\neq0$. On $\partial\mathcal X$ since
$\nabla\psi=n$, from~\eqref{eq: non-characteristic with boundary}, we have $\sum_i(V_i\psi)^2=\epsilon c^2\geq\epsilon\kappa^2>0$ and $r\cdot n=c\geq\kappa>0$.
Thus the boundary is uniformly non-characteristic
and the reflection is transversal.
Assumption~\ref{ass:hormander} and compactness implies uniform restricted H\"ormander condition.
The boundary is very good by
\cite[p.~565]{cattiaux1992stochastic}.
The process is nonsticky, so its boundary-viscosity
coefficient is zero, 
Hence there exists a  Lebesgue transition density $k^r_{t,\epsilon}\in C^\infty(\overline{\mathcal X}\times \overline{\mathcal X})$
for $t>0$. Assumption~\ref{ass:hormander} implies that $\overline{\mathcal X}$ is connected. Hence, strict positivity follows from
\cite[proof of Theorem~5.36, p.~610]
{cattiaux1992stochastic}.
From~\cite[pp.~567-568]{agrachev2019comprehensive}, since $d\mathrm m_{\mathrm P}(x)=\omega_{\mathrm P}(x)dx$, where $\omega_{\mathrm P}>0$ and $\omega_{\mathrm P}\in C^{\infty}(\overline{\mathcal{X}})$, we have that $p^r_{t,\epsilon}(x,y)=\frac{k^r_{t,\epsilon}(x,y)}{\omega_{\mathrm P}(y)}$ retains smoothness and strict positivity.
\end{proof}
We state our first main result.
\begin{theorem}\label{thm:properness}
Let $d\mu_0=\rho_0\,d\mathrm{m}_{\mathrm P}$ and $d\mu_f=\rho_f\,d\mathrm{m}_{\mathrm P}$ and let $p^r$ be the transition density function in Proposition~\ref{lem:kernel}. An optimal control for Problem~$2$ is: 
\begin{equation}\label{eq:optimal_feedback}
u_r(t,x)=\epsilon\,g(x)^\top\nabla\log\varphi_r(t,x),
\end{equation}
where  $$\varphi_r(t,x):=\mathbb{E}_{\mathrm{R}^{\epsilon,\mu_0}}\!\left[
\phi_f(X_{1})
\,\middle|\,
X_t=x
\right]=\int_{\overline{\mathcal{X}}} \phi_f(y)p^r_{1-t,\epsilon}(x,y)\,d\mathrm m_{\mathrm P}(y)$$ and
$(\phi_0,\phi_f)$ are positive functions on $\overline{\mathcal{X}}$ that satisfy $$\phi_0(x)\int_{\overline{\mathcal{X}}} \phi_f(y)p^r_{1,\epsilon}(x,y)\,d\mathrm m_{\mathrm P}(y)=1,\quad\text{for $\mu_0$-a.e $x$}$$ and $$\phi_f(y)\int_{\overline{\mathcal{X}}} \phi_0(x)p^r_{1\epsilon}(x,y)\rho_0(x)\,d\mathrm m_{\mathrm P}(x)=\rho_f(y),\quad\text{ for $\mathrm m_{\mathrm P}$-a.e. $y$.} $$
\end{theorem}
\begin{proof}
Since $\overline{\mathcal X}$ is compact, by Fortet's theorem~\cite{Fortet1940}, Proposition~\ref{lem:kernel} and Assumption~\ref{eq: regularity_on_density} implies the existence of a positive, continuous and bounded Schr\"odinger factors
$(\phi_0,\phi_f)$. Therefore,
$$d\mathrm{P}^\star:=\phi_0(X_0)\phi_f(X_1)d\mathrm{R}^{\epsilon,\mu_0}\in \mathcal{C}_{\mathrm{SB}}(\mathrm{R}^{\epsilon,\mu_0},\mu_f)$$ and $\KL(\mathrm{P}^\star\mid \mathrm{R}^{\epsilon,\mu_0})<\infty$.
Let $$\mathrm{P}\in \mathcal{C}_{\mathrm{SB}}(\mathrm{R}^{\epsilon,\mu_0},\mu_f).$$ Since $$\KL(\mathrm{P}\mid \mathrm{R}^{\epsilon,\mu_0})-\KL(\mathrm{P}^\star\mid \mathrm{R}^{\epsilon,\mu_0})
=\KL(\mathrm{P}\mid \mathrm{P}^\star)\ge0$$ and $\mathrm{P}\mapsto\KL(\mathrm{P}\mid \mathrm{R}^{\epsilon,\mu_0})$ is strictly convex, we have that $\mathrm{P}^\star$ is the unique minimizer for Problem~$1$. Since $\varphi_r$ in~\eqref{eq:optimal_feedback} is the mild solution to $$\partial_t\varphi_r + \mathcal L_\epsilon \varphi_r =0,\quad\text{  with $\varphi_r(1,\cdot)=\phi_f(\cdot)$,}$$
 where $\mathcal L_\epsilon$ is in~\eqref{eq: infinitesimal_generator}, by It\^o's formula, $$M_t
:=\frac{\varphi_r(t,X_t)}{\varphi_r(0,X_0)}
=\mathbb E_{\mathrm{R}^{\epsilon,\mu_0}}[\phi_0(X_0)\phi_f(X_1)\mid\mathcal F_t]\quad\text{satisfies}\quad  dM_t
=\sqrt{\epsilon}\,M_t
(g^\top\nabla\log\varphi_r)(t,X_t)dW_t$$ and hence
Girsanov changes the drift of~\eqref{eq:reference_SDE} to $b_\epsilon+gu_r$,
with $u_r$ in~\eqref{eq:optimal_feedback}. For $b<1$, It\^o's formula under $\mathrm{P}^\star$ yields $$\mathbb E_{\mathrm{P}^\star}[\log M_b]
=\frac{1}{2\epsilon}\mathbb E_{P^\star}
  \int_0^b\|u_r(t,X_t)\|^2\,dt.$$ Since $M_t$ is uniformly bounded above and away from zero and $M_t\mapsto\phi_0(X_0)\phi_f(X_1)$ as $t\mapsto 1$,
dominated convergence on the left and monotone convergence
on the right of $$\mathbb E_{\mathrm{P}^\star}[\log M_b]
=\frac{1}{2\epsilon}\mathbb E_{P^\star}
  \int_0^b\|u_r(t,X_t)\|^2\,dt,$$ as $b\uparrow1,$ give $$\KL(\mathrm{P}^\star\mid \mathrm{R}^{\epsilon,\mu_0})
=\frac{1}{2\epsilon}\mathbb E_{P^\star}
  \int_0^1\|u_r(t,X_t)\|^2\,dt<\infty.$$ 
Hence, ~\eqref{eq:optimal_feedback} minimizes Problem~2.
\end{proof}

\section{Convex Eulerian Formulation}\label{sec:Eulerian}
Since explicit closed-form formulas for $p^r$ in Proposition~\ref{lem:kernel} are
generally unavailable~\cite{hsu1986excursions,lou2023reflected}, the goal of this section is to facilitate a numerical scheme through PDEs. To this end, following from~\cite[pp.~567-568]{agrachev2019comprehensive}, throughout we write $d\mathrm m_{\mathrm P}(x)=\omega_{\mathrm P}(x)dx$, where $\omega_{\mathrm P}>0$ and $\omega_{\mathrm P}\in C^{\infty}(\overline{\mathcal{X}})$. Let $m_{\epsilon,\mathrm{ctl}}:=\rho\,u$, where $\rho$ is a density function of~\eqref{eq:controlled_reflected_SDE}. Then inspired by~\cite{BenamouBrenier2000}, the Eulerian problem associated to Problem~$2$ is 
\begin{equation*}
\inf_{(\rho,m_{\epsilon,\mathrm{ctl}})} \mathcal{J}(\rho, m_{\epsilon,\mathrm{ctl}}):=\int_0^{1}\int_{\overline{\mathcal{X}}} \frac{\norm{m_{\epsilon,\mathrm{ctl}}(t,x)}^2}{2\rho(t,x)}dm_{\mathrm P}d t
\end{equation*}
subject to
\begin{equation}\label{eq:VA_FPK}
\p_t \rho
+ \frac{1}{\omega_{\mathrm P}}\diver\!\bigl(b_\epsilon\rho\omega_{\mathrm P} + g m_{\epsilon,\mathrm{ctl}}\omega_{\mathrm P}\bigr)
- \frac{\epsilon}{2\omega_{\mathrm P}}\sum_{i,j=1}^d \p_{ij}(G_{ij}\rho\omega_{\mathrm P})=0
\quad\text{in }(0,1)\times \mathcal{X},
\end{equation}
\begin{equation}\label{eq:VA_Eulerian_endpoints}
\rho(0,\cdot)=\rho_0,\qquad \rho(1,\cdot)=\rho_f,
\end{equation}
with no flux condition 
\begin{equation}\label{eq:VA_no_flux}
\Bigl(
b_\epsilon\rho\omega_{\mathrm P} + g m_{\epsilon,\mathrm{ctl}}\omega_{\mathrm P} - \frac{\eps}{2}\,\diver(G\rho\omega_{\mathrm P})
\Bigr)\cdot n =0
\end{equation}
on $(0,1)\times \p\mathcal{X}$. Here  $(\diver(G\rho\omega_{\mathrm P}))_i:=\sum_{j=1}^d \p_j(G_{ij}\rho\omega_{\mathrm P})$ and the solution $(\rho,m_{\epsilon,\mathrm{ctl}})$ is to be understood in the weak sense with the test functions $\zeta\in C([0,1]\times\overline{\mathcal X})\cap
C^{1,2}((0,1)\times\overline{\mathcal X})$ satisfying $r\cdot\nabla \zeta=0$ on $(0,1)\times \p\mathcal X$.
We characterize the optimal solution. 
\begin{theorem}\label{prop: Strong duality and attainment}
Let $(\varphi_\epsilon,\hat\varphi_\epsilon)\in
 C([0,1]\times\overline{\mathcal X})\cap C^{1,2}((0,1)\times\overline{\mathcal X})$ be strictly
positive mild solutions of
\begin{align}
 &\partial_t\varphi_\epsilon + \mathcal L_\epsilon \varphi_\epsilon =0,\quad \partial_t\hat\varphi_\epsilon - \mathcal L_{\epsilon,\mathrm m_{\mathrm P}}^* \hat\varphi_\epsilon =0,\label{eq:Schr-bwd}\\ 
&\varphi_\epsilon(0,\cdot)\hat\varphi_\epsilon(0,\cdot)=\rho_0,\quad
\varphi_\epsilon(1,\cdot)\hat\varphi_\epsilon(1,\cdot)=\rho_f,\label{eq:Schr-bc}   
\end{align}
with continuous endpoint traces. Here, $\mathcal L_\epsilon$ is in~\eqref{eq: infinitesimal_generator} with domain $\mathrm{Dom_r}$ in~\eqref{eq: domain_infinitesimal_generator} and $$\mathcal L_{\epsilon,\mathrm m_{\mathrm P}}^*\hat\varphi=-\frac{1}{\omega_{\mathrm P}}\nabla\cdot (b_\epsilon\hat\varphi\omega_{\mathrm P})
+\frac{\eps}{2\omega_{\mathrm P}}\sum_{i,j=1}^d \p_{ij} (G_{ij}\hat\varphi\omega_{\mathrm P})$$  with domain $\mathrm{Dom}^*_{n.\mathrm m_{\mathrm P}}:= \bigl\{\hat\varphi \in C^2(\overline{\mathcal{X}}) : \bigl(b_\epsilon \hat\varphi\omega_{\mathrm P} - \frac{\epsilon}{2}\mathrm{div}(G\hat\varphi\omega_{\mathrm P})\bigr) \cdot n = 0 \text{ on } \partial\mathcal{X}\bigr\}$. 
Then: 
\begin{equation*}
\inf_{\substack{(\rho_{\epsilon},  m_{\epsilon,\mathrm{ctl}}) \text{ in \eqref{eq:VA_FPK}-\eqref{eq:VA_no_flux}}}} 
\mathcal{J}(\rho_{\epsilon},  m_{\epsilon,\mathrm{ctl}}) = \sup_{\lambda_{\epsilon}\in\Lambda}\mathcal F(\lambda_{\epsilon}):=\int_{\overline{\mathcal{X}}}\lambda(1,x)\rho_f(x) d\mathrm{m}_{\mathrm P}
-\int_{\overline{\mathcal{X}}}\lambda(0,x)\rho_0(x) d\mathrm{m}_{\mathrm P},    
\end{equation*}
where 
\begin{multline*}
\Lambda := \big\{ 
\lambda\in C([0,1]\times\overline{\mathcal X})\cap
C^{1,2}((0,1)\times\mathcal X);~\lambda(t,\cdot)\in\mathrm{Dom}_r
\quad\text{for }0<t<1, \\ ~\partial_t\lambda+\mathcal L_\epsilon\lambda
+\tfrac12\|g^{\top}\nabla\lambda\|^{2} \le 0 \quad \text{on } (0,1)\times\mathcal X \big\}.   
\end{multline*} 
The supremum is attained at
$\lambda_\epsilon^{\star}=\epsilon\log\varphi_\epsilon$ and the primal infimum is uniquely
attained at
$\rho_\epsilon^{\star}=\varphi_\epsilon\hat\varphi_\epsilon$ and
$m_{\epsilon,\mathrm{ctl}}^{\star}=\epsilon\hat\varphi_\epsilon
g^{\top}\nabla\varphi_\epsilon$. The optimal control for Problem~$2$ is $u_\epsilon^\star=\frac{m^{\star}_{\epsilon,\mathrm{ctl}}}{\rho_\epsilon^\star}
=
\epsilon\,g^\top \nabla \log \varphi_\epsilon$.
\end{theorem}
Before we provide a proof, note that $\mathrm{Dom}^*_{n.\mathrm m_{\mathrm P}}$ decomposes into  
\begin{equation}\label{expand_currernt}
r\cdot\nabla\hat\varphi
+
\left(
r\cdot\nabla\log\omega_{\mathrm P}
+
\frac{
\left(\operatorname{div}G-\frac{2}{\epsilon}b_\epsilon\right)
\cdot n
}{r\cdot n}
\right)\hat\varphi=0.
\end{equation}
If $g\equiv I_d$ ($m=d$) and $\omega_{\mathrm P}\equiv 1$ then,  $b_\epsilon=0$, $\operatorname{div}G=0$ and $r=n$. Thus, the boundary conditions for~\eqref{eq:Schr-bwd}-\eqref{eq:Schr-bc} reduce to the symmetric zero-drift pair $n\cdot\nabla\varphi=0$, $n\cdot\nabla\hat\varphi=0$ of~\cite{caluya2021reflected}.

\begin{proof}
For any feasible pair $(\rho,m_{\epsilon,\mathrm{ctl}})$ and $\lambda\in\Lambda$, testing on
$[a,b]\subset(0,1)$ and completing the square gives $$\int_a^b\!\!\int_{\overline{\mathcal X}}
\frac{|m_{\epsilon,\mathrm{ctl}}|^2}{2\rho}\,d\mathrm m_{\mathrm P}\,dt
-\left[\int_{\overline{\mathcal X}}\lambda\rho
\,d\mathrm m_{\mathrm P}\right]_a^b
=
\int_a^b\!\!\int_{\overline{\mathcal X}}
\left[
\frac{|m_{\epsilon,\mathrm{ctl}}-\rho g^\top\nabla\lambda|^2}{2\rho}
-\rho H_\lambda
\right]d\mathrm m_{\mathrm P}\,dt\geq0,$$ where $$H_\lambda:=\partial_t\lambda+\mathcal L_\epsilon\lambda
+\frac12|g^\top\nabla\lambda|^2.$$
The boundary terms vanish by no flux and
$r\cdot\nabla\lambda=0$.
Letting $a\downarrow0$, $b\uparrow 1$, using the
endpoint traces, yields
$\mathcal F(\lambda)\leq\mathcal J(\rho,m_{\epsilon,\mathrm{ctl}})$. This implies that strong duality is achieved if $m_{\epsilon,\mathrm{ctl}}=\rho g^\top\nabla\lambda$ and $H_{\lambda}=0$.
Set $\lambda_\epsilon^\star=\epsilon\log\varphi_\epsilon$. Then, $\lambda_\epsilon^{\star}\in
 C([0,1]\times\overline{\mathcal X})\cap C^{1,2}((0,1)\times\overline{\mathcal X})$. Also, 
 \[
 r\cdot\nabla\lambda_\epsilon^{\star}
=\epsilon\,\varphi_\epsilon^{-1}r\cdot\nabla\varphi_\epsilon=0\quad\text{ on
$(0,1)\times\partial\mathcal X$}
\]
and finally, from~\cite[\S~4.4]{evans2022partial}, we get that $H_{\lambda^{\star}_\epsilon}=0$. Secondly, let $\rho_\epsilon^{\star}=\varphi_\epsilon\hat\varphi_\epsilon$. Since $\varphi\in\mathrm{Dom_r}$ and $\hat\varphi\in\mathrm{Dom}^*_{n.\mathrm m}$, using $$\operatorname{div}(G\varphi\hat\varphi\omega_{\mathrm P})
=\varphi\operatorname{div}(G\hat\varphi\omega_{\mathrm P})
+\hat\varphi\,\omega_{\mathrm P}\,G\nabla\varphi,$$ we get that $$\left(b_\epsilon\rho_\epsilon^{\star}\omega_{\mathrm P}
+g m^{\star}_{\epsilon,\mathrm{ctl}}\omega_{\mathrm P}
-\frac{\epsilon}{2}\operatorname{div}(G\rho_\epsilon^{\star}\omega_{\mathrm P})\right)\cdot n
=0.$$ Hence, from the continuous endpoint traces in~\eqref{eq:Schr-bwd}-\eqref{eq:Schr-bc}, we get that $\int_{\overline{\mathcal X}} \rho_\epsilon^{\star}\,d\mathrm m_{\mathrm P}=1$, for all $t\in[0,1]$. For the interior, using $$\mathcal L_{\epsilon,\mathrm m_{\mathrm P}}^*(\varphi\hat\varphi)=\varphi\mathcal L_{\epsilon,\mathrm m_{\mathrm P}}^*\hat\varphi-\hat\varphi\,\mathcal L_\epsilon\varphi
+\frac{\epsilon}{\omega_{\mathrm P}}\operatorname{div}\!\left(\hat\varphi\,\omega_{\mathrm P} G\nabla\varphi\right),$$ we get that $$\partial_t\rho_\epsilon^{\star}
=\mathcal L_{\epsilon,\mathrm m}^{*}\rho_\epsilon^{\star}
-\frac{1}{\omega_{\mathrm P}}\operatorname{div}(g m^{\star}_{\epsilon,\mathrm{ctl}}\omega_{\mathrm P}),$$ which is exactly~\eqref{eq:VA_FPK}.  Since
$\lambda_\epsilon^\star$ is bounded and continuous and
$m_{\epsilon,\mathrm{ctl}}^\star
=\rho_\epsilon^\star g^\top\nabla\lambda_\epsilon^\star$, by continuous endpoint traces and monotone convergence, we get $\mathcal F(\lambda_\epsilon^\star)=\mathcal J(\rho_\epsilon^\star,
m_{\epsilon,\mathrm{ctl}}^\star)<\infty$. 

For uniqueness, any minimizer $(\rho_\epsilon,m_{\epsilon,\mathrm{ctl}})$ satisfies
$m_{\epsilon,\mathrm{ctl}}
=\epsilon\rho_\epsilon g^\top\nabla\log\varphi_\epsilon$ and hence $ \partial_t(\rho_\epsilon/\varphi_\epsilon) - \mathcal L_{\epsilon,\mathrm m_{\mathrm P}}^* (\rho_\epsilon/\varphi_\epsilon) =0$ with $\rho(0,\cdot)/\varphi_\epsilon(0,\cdot)=\rho_0/\varphi_\epsilon(0,\cdot)=\hat \varphi_\epsilon(0,\cdot)$. Given $\hat \varphi_\epsilon(0,\cdot)$, since the solution of~\eqref{eq:Schr-bwd} is unique,  we get
$\rho/\varphi_\epsilon=\hat\varphi_\epsilon$ which implies $\rho_\epsilon=\rho^\star_\epsilon$.
\end{proof}

\begin{figure}[t]
\centering
\begin{tikzpicture}[
    >=Stealth,
    font=\small,
    line width=0.5pt
]
    \node (hatphi0) at (0,0)
        {$\hat{\varphi}_{\epsilon,0}$};
    \node (hatphif) at (3,0)
        {$\hat{\varphi}_{\epsilon,1}$};
    \node (phif) at (3,-1.6)
        {$\varphi_{\epsilon,1}$};
    \node (phi0) at (0,-1.6)
        {$\varphi_{\epsilon,0}$};

    \draw[->] (hatphi0) -- (hatphif)
        node[midway,above=3pt,align=center] {
            $\displaystyle
            r\cdot\nabla\hat\varphi_\epsilon
+
\left[
r\cdot\nabla\log\omega_{\mathrm P}
+
\frac{
\left(\operatorname{div}G-\frac{2}{\epsilon}b_\epsilon\right)
\cdot n
}{r\cdot n}
\right]\hat\varphi_\epsilon=0$
            \\[2pt]
            
        }
        node[midway,below=2pt,font=\scriptsize]
            {$\hat{\varphi}_{\epsilon,t}=e^{t\mathcal L_{\epsilon,\mathrm m_{\mathrm P}}^*}\hat{\varphi}_{\epsilon,0}$};

    \draw[->] (hatphif) -- (phif)
        node[midway,right=3pt]
            {$\rho_f\oslash\hat{\varphi}_{\epsilon,1}$};

    \draw[->] (phif) -- (phi0)
        node[midway,below=3pt,align=center] {
            
            $r\cdot\nabla\varphi_\epsilon=0$
            \\[2pt]
        }
        node[midway,above=2pt,font=\scriptsize]
            {$\varphi_{\epsilon,t}=e^{(1-t)\mathcal L_\epsilon}\varphi_{\epsilon,1}$};

    \draw[->] (phi0) -- (hatphi0)
        node[midway,left=3pt]
            {$\rho_0\oslash\varphi_{\epsilon,0}$};
\end{tikzpicture}
\caption{%
Mild semigroup iteration for~\eqref{eq:Schr-bwd}-\eqref{eq:Schr-bc}. Nodes denote endpoint traces; pointwise divisions  $\oslash$ enforce both
endpoint couplings at a fixed point. The boundary conditions are generally asymmetric and determined by the oblique reflection, reference measure, sub-Riemannian diffusion coefficients, and the Stratonovich-to-It\^o correction.}
\label{fig:reflected_Schrodinger_recursion}
\end{figure}
\section{Numerical Example}\label{sec:Numerical}
We demonstrate our theoretical result with an example. 

\begin{figure}[!t]
\centering
\includegraphics[width=0.25\columnwidth]{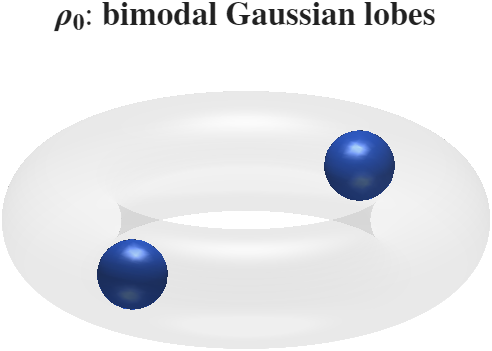}%
\hspace{0.2\columnwidth}%
\includegraphics[width=0.25\columnwidth]{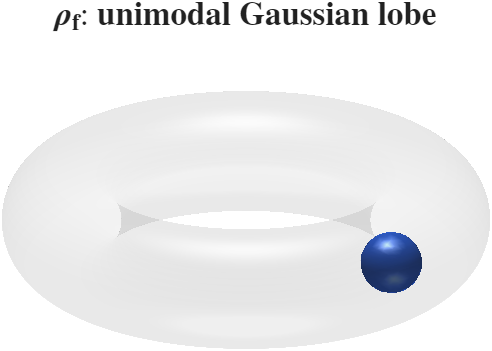}
\caption{Initial (left) and final (right) density isosurfaces
from~\eqref{eq:initial_density}--\eqref{eq:final_density}.}
\label{fig:boundary_densities}
\end{figure}
\begin{example}
\label{ex:unified_example}
Consider $(\overline {\mathcal{B}}\times S^1,\mathrm D,\mathrm d)$ where \(\overline {\mathcal{B}}=\{(x_1,x_2)\in\mathbb R^2: x_1^2+x_2^2\leq 1\}\) and $S^1$ is the unit circle parameterized by \(\theta\in[0,2\pi)\). Fix $\alpha>0$ and
$0<\delta<\sqrt{\alpha^2+4\alpha}-\alpha $ and for $x=(x_1,x_2,\theta)\in\overline {\mathcal{B}}\times S^1$, the horizontal subbundle
\begin{equation}\label{eq: example_horizontal}
\mathrm D(x) := \operatorname{Span}\left\{
\begin{pmatrix}1\\0\\h_1(x)\end{pmatrix},
\begin{pmatrix}0\\1\\h_2(x)\end{pmatrix}
\right\},
\end{equation}
where $h_1(x)=2\alpha x_2+\delta\cos\theta$ and $
h_2(x)=-2\alpha x_1+\delta\sin\theta$. We endow $\mathrm D$ with the metric $\mathrm d_x(v,v) = v^\top G(x)^\dagger v$, for \(v\in \mathrm D(x)\) and $x\in\overline{\mathcal{X}}$, where
\begin{equation}\label{eq: diffusion_mat_value_normal}
G=gg^\top=
\begin{pmatrix}
1&0&h_1\\
0&1&h_2\\
h_1&h_2&h_1^2+h_2^2
\end{pmatrix}.
\end{equation}
Throughout, we set $q(x):=x_1\cos\theta+x_2\sin\theta$ and $s(x):=x_1\sin\theta-x_2\cos\theta$. Then, from~\cite[Theorem~20.6]{agrachev2019comprehensive}, we have that $\omega_{\mathrm P}(x)=\frac{1}{|B(x)|}$,
up to normalization of the Haar measure on \(S^1\), where $B(x):=-4\alpha-2\alpha\delta s(x)+\delta^2<0$. 
Since the defining function $\xi(x_1,x_2,\theta)=1-x_1^2-x_2^2$, the inward unit normal is $n(x)=(-x_1,-x_2,0)$, for all \(x\in\partial\mathcal X=S^1\times S^1\). Hence, since
\begin{equation}\label{eq: denom_oblique}
g(x)^\top n(x)=
\begin{pmatrix}
-x_1\\
-x_2
\end{pmatrix}
\text{ and }
\lVert g(x)^\top n(x)\rVert=1,
\end{equation}
$S^1\times S^1$ is noncharacteristic. Since~\eqref{eq: reflection field} simplifies to
\begin{equation}
\label{eq:example_oblique_reflection}
r(x)
=
\begin{pmatrix}
-x_1\\
-x_2\\
-\delta q
\end{pmatrix}=-x_1g_1(x)-x_2g_2(x)\in\mathrm D(x)
\end{equation}
for all \(x\in S^1\times S^1\) and $r(x)\cdot n(x)=1>0$, we have that  $r$ is uniformly inward pointing. Also, since 
$r=n-\delta q\,e_\theta$ where  $e_\theta=(0,0,1)$, we have that $r$
is oblique relative to $n$ wherever $q\neq0$
and $r=n$ on $\{q=0\}$. Thus, $r$ in~\eqref{eq:example_oblique_reflection} is genuinely oblique.


Given
\begin{equation}\label{eq:initial_density}
\rho_0(x)=\frac{1}{2C_0}
\sum_{k=1}^{2}
\exp\left(
-\frac{\|F(x)-c_k\|^2}{2\sigma_0^2}
\right)
\end{equation}
and
\begin{equation}\label{eq:final_density}
\rho_f(x)=\frac{1}{C_f}
\exp\left(
-\frac{\|F(x)-c_f\|^2}{2\sigma_f^2}
\right).
\end{equation}
Here $C_0:=\frac12\sum_{k=1}^{2}
\int_{\overline{\mathcal B}\times S^1}
\exp\left(
-\frac{\|F(x)-c_k\|^2}{2\sigma_0^2}
\right)\,d\mathrm m_{\mathrm P}$, $C_f:=\int_{\overline{\mathcal B}\times S^1}
\exp\left(
-\frac{\|F(x)-c_f\|^2}{2\sigma_f^2}
\right)\,d\mathrm m_{\mathrm P}$, $F(x)=\big((R+ax_1)\cos\theta,\,
(R+ax_1)\sin\theta,\,ax_2\big)^\top$,
$c_1=(0,R,0)^\top$, $c_2=(0,-R,0)^\top$,
$c_f=(R,0,0)^\top$, $R=2$, $a=0.65$,
$\sigma_0=0.23$, and $\sigma_f=0.20$.
The task is to solve the SBP~\eqref{eq: horizontal subbundle}-\eqref{eq: into_minimum_energy} over $(\overline {\mathcal{B}}\times S^1,\mathrm D,\mathrm d)$ with reflection mechanism~\eqref{eq:example_oblique_reflection}.

We verify the assumptions as follows. For Assumption~\ref{eq: regularity_on_density}, we verify that $\rho_0,\rho_f\in C^{\infty}(\overline {\mathcal{B}}\times S^1;\R_{>0})$. The vector fields in~\eqref{eq: example_horizontal} satisfy Assumption~\ref{ass:regularity}. For Assumption~\ref{ass:hormander},  since $[g_1,g_2]
=
B(x)\,\partial_\theta$, we have that $\operatorname{span}\{g_1,g_2,[g_1,g_2]\}
=\operatorname{span}\{\partial_{x_1},\partial_{x_2},
\partial_\theta\}
=
\mathbb R^3$
holds, at every point $x\in\overline{\mathcal X}$. Finally, for Assumption~\ref{ass:obliqueness}, from~\eqref{eq: denom_oblique}, we have $\|g(x)^\top n(x)\|^2
=
1>0$, on \( S^1\times S^1\). 

Therefore, the optimal control is $u_\epsilon^\star
=
\epsilon\,g^\top \nabla \log \varphi_\epsilon$, where $(\varphi_\epsilon,\hat\varphi_\epsilon)$ solve~\eqref{eq:Schr-bwd}-\eqref{eq:Schr-bc}, 
with
\begin{equation*}
\mathcal L_\epsilon \varphi=
-\epsilon\alpha\delta q\,\varphi_\theta
+\frac{\epsilon}{2}
\Big[
\varphi_{x_1x_1}+\varphi_{x_2x_2}
+2h_1\varphi_{x_1\theta}
+2h_2\varphi_{x_2\theta}
+(h_1^2+h_2^2)\varphi_{\theta\theta}
\Big]
\end{equation*}
and
\begin{equation*}
\mathcal L_{\epsilon,\mathrm m_{\mathrm P}}^*\hat\varphi=-|B|\nabla\cdot \left(b_\epsilon\frac{\hat\varphi}{|B|}\right)
+\frac{\eps|B|}{2}\sum_{i,j=1}^d \p_{ij} \left(G_{ij}\frac{\hat\varphi}{|B|}\right)    
\end{equation*}
accompanied by the asymmetric boundary conditions
\[
r\cdot\nabla\varphi=0~~\text{and}~~ r\cdot\nabla\hat\varphi + \left( \delta s-r\cdot\nabla\log|B| \right)\hat\varphi =0
\]
on $S^1\times S^1$. Since $r\cdot\nabla\rho_f+
(\delta s-r\cdot\nabla\log|B|)\rho_f\not\equiv 0$ on $S^1\times S^1$, we use the mild formulation
with continuous endpoint traces in~\eqref{eq:Schr-bwd}-\eqref{eq:Schr-bc} through the numerical scheme in Figure~\ref{fig:reflected_Schrodinger_recursion}. Once $\varphi_\epsilon$ is obtained, the optimal control $u_\epsilon^\star
=
\epsilon\,g^\top \nabla \log \varphi_\epsilon$ steers~\eqref{eq:controlled_reflected_SDE} with~\eqref{eq:example_oblique_reflection} in $(\overline {\mathcal{B}}\times S^1,\mathrm D,\mathrm d)$ from \(\rho_0\) in~\eqref{eq:initial_density} to \(\rho_f\) in~\eqref{eq:final_density} (see Figure~\ref{fig:controlled_oblique_reflected_SDE}).

\begin{figure}[!t]
\centering
\includegraphics[width=0.45\textwidth]
{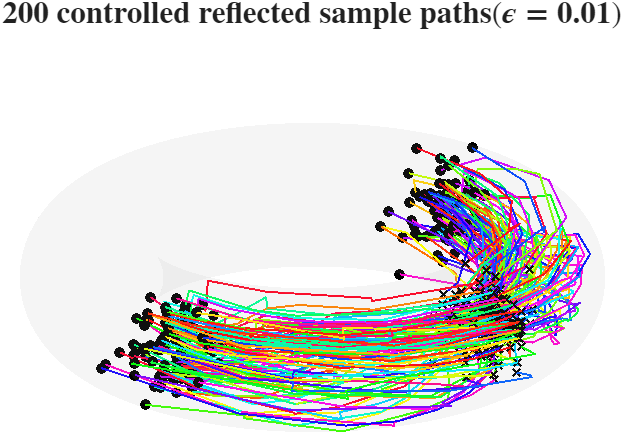}
\caption{
Sample paths of the optimally controlled oblique reflected diffusion in
\(\overline {\mathcal{B}}\times S^1\) with velocity constrained in~\eqref{eq: example_horizontal} with $\alpha=\delta=1$ and $\epsilon=0.01$. The feedback \(u_\epsilon^\star\) steers  from the
initial samples to the  final samples prescribed in Figure~\ref{fig:boundary_densities}.
}
\label{fig:controlled_oblique_reflected_SDE}
\end{figure}

\begin{figure}[!t]
\centering
\includegraphics[width=0.98\textwidth]
{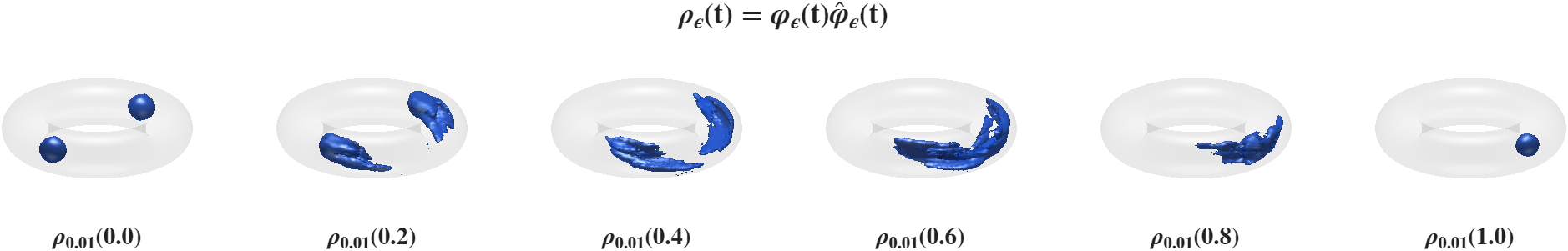}
\caption{
Snapshots of density evolution $t\mapsto\rho_\epsilon^{\star}(t,\cdot)=\varphi_\epsilon(t,\cdot)\hat\varphi_\epsilon(t,\cdot)$, where $t=0,0.2,0.4,0.6,0.8,1$, for Figure~\ref{fig:controlled_oblique_reflected_SDE} obtained using Sinkhorn in Figure~\ref{fig:reflected_Schrodinger_recursion}.
}
\label{fig:all_eps_morphing}
\end{figure}
\begin{figure}[t]
\centering
\includegraphics[width=0.4\columnwidth]
{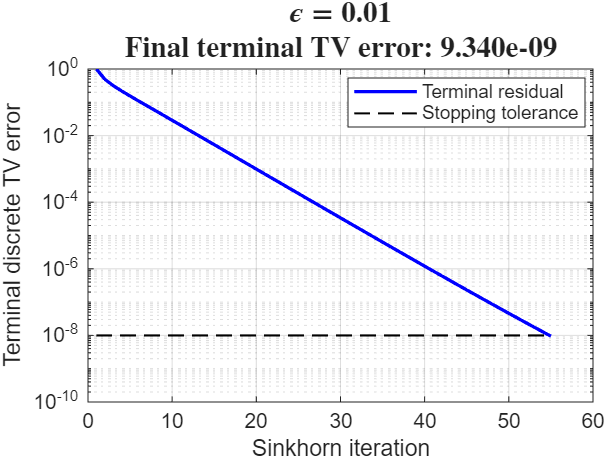}
\hspace{0.02\columnwidth}%
\includegraphics[width=0.4\columnwidth]
{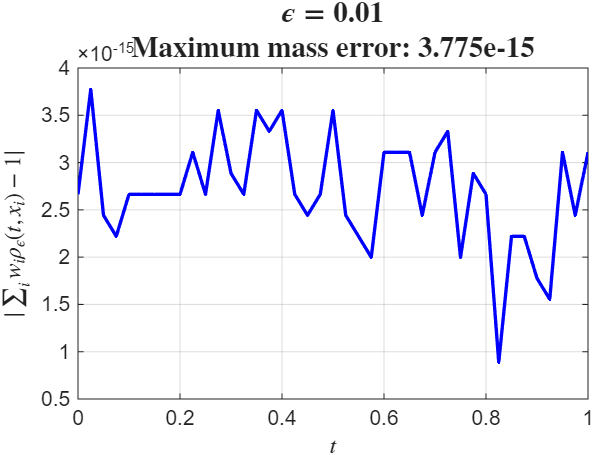}
\caption{Validation of the density evolution in
Figure~\ref{fig:all_eps_morphing}: The terminal total-variation error on the computational grid decreases below the
stopping tolerance $10^{-8}$ (left);
near-zero mass-conservation error throughout the evolution (right).}
\label{fig:Sinhkorn_convergence}
\end{figure}
\end{example}

\begin{example}
Consider $(\overline {\mathcal{B}}\times S^1,\mathrm D,\mathrm d)$. Then the
inward Euclidean unit normal is $n(x)=(-x_1,-x_2,0)$. Given $$\mathrm D(x) := \operatorname{Span}\left\{
\begin{pmatrix}\cos(\theta)\\\sin(\theta)\\0\end{pmatrix},
\begin{pmatrix}0\\0\\1\end{pmatrix}
\right\},$$ 
we get that $$g(x)^\top n(x)
=
\begin{pmatrix}
-x_1\cos\theta-x_2\sin\theta\\
0
\end{pmatrix},$$
for all \(x=(x_1,x_2,\theta)\in S^1\times S^1\). By reparameterizing $S^1\times S^1$ as $x=(\cos\phi,\sin\phi,\theta)$,
we obtain $$g(x)^\top n(x)
=
\begin{pmatrix}
-\cos(\theta-\phi)\\
0
\end{pmatrix}.$$
Therefore, if $\theta=\phi+\frac{\pi}{2}$ or $\theta=\phi-\frac{\pi}{2}$, 
then $g(x)^\top n(x)
=
\begin{pmatrix}
0\\
0
\end{pmatrix}$ and~\eqref{eq: reflection field} is indeterminate at characteristic points and our results do not apply.
\end{example}

\section{Conclusion}\label{sec: conclusion}
We formulated a reflected Schr\"odinger bridge problem on sub-Riemannian domains with non-characteristic
boundaries. The intrinsic horizontal reflection leads
to a mild Schr\"odinger system with asymmetric boundary
conditions and continuous endpoint traces, enabling
a PDE-based Sinkhorn iteration. This non-characteristic
boundary restriction is structural (due to the construction from optimization) and our future direction will be to allow
sticky boundary dynamics, characteristic boundary configurations~\cite{wentzell1959boundary}. Other current works in progress include
generative modeling~\cite{deng2024reflected} and flow matching~\cite{adu2025flow,chen2024flow}, small-noise optimal
transport limits, and unmatched control and noise channels~\cite{bondar2026nonlinear}.

\bibliographystyle{IEEEtran}
\bibliography{references} 
\end{document}